\RequirePackage{amsmath}
\documentclass[dvipsnames,runningheads]{llncs}
\usepackage[T1]{fontenc}
\usepackage{lmodern}
\usepackage{graphicx}
\usepackage{commands}
\usepackage[hidelinks]{hyperref}
\usepackage{color}
\usepackage[export]{adjustbox}

\begin{document}
\title{Flow Matching on Lie Groups}
%
%
\author{
Finn M. Sherry\inst{1}
\and
Bart M.N. Smets\inst{1}
}
\authorrunning{F.M. Sherry \& B.M.N. Smets}
%
\institute{
CASA \& EAISI, Dept. of Mathematics \& Computer Science, Eindhoven University of Technology, the Netherlands\\
\email{\{f.m.sherry,b.m.n.smets\}@tue.nl}
}
\maketitle

\begin{abstract}
Flow Matching (FM) is a recent generative modelling technique: we aim to learn how to sample from distribution $\dist{X}_1$ by flowing samples from some distribution $\dist{X}_0$ that is easy to sample from. 
The key trick is that this flow field can be trained while conditioning on the end point in $\dist{X}_1$: given an end point, simply move along a straight line segment to the end point \cite{Lipman2023FlowModeling}. 
However, straight line segments are only well-defined on Euclidean space. 
Consequently, 
\cite{Chen2024FlowGeometries} generalised the method to FM on Riemannian manifolds, replacing line segments with geodesics or their spectral approximations.
We take an alternative point of view: we generalise to FM on Lie groups with surjective exponential maps by instead substituting exponential curves for line segments. This leads to a simple, intrinsic, and fast implementation for many matrix Lie groups, since the required Lie group operations (products, inverses, exponentials, logarithms) are simply given by the corresponding matrix operations. 
FM on Lie groups could then be used for generative modelling with data consisting of sets of features (in $\R^n$) and poses (in some Lie group), e.g. the latent codes of Equivariant Neural Fields \cite{Wessels2025GroundingFields}.

\keywords{Flow Matching \and Lie Groups \and Exponential Curves \and Generative Modelling}
\end{abstract}

\section{Introduction}\label{sec:introduction}
The aim of generative modelling is to learn how to sample from distribution $\dist{X}$, given a large data set of samples. Chen et al. \cite{Chen2018NeuralEquations} proposed learning a flow from some distribution $\dist{X}_0$ that is easy to sample from, e.g. white noise, to the target distribution $\dist{X} \eqqcolon \dist{X}_1$.
Concretely, we look for a smooth flow
$\psi: [0, 1] \to \Diff(\R^d)$
such that $\dist{X}_1 = (\psi_1)_\# \dist{X}_0$, with $\#$ the measure push-forward. We can then define intermediate distributions $\dist{X}_t \coloneqq (\psi_t)_\# \dist{X}_0$.
Such a flow is induced by a time dependent smooth vector field
$u: [0, 1] \to \sections(T\R^d)$,
satisfying $\partial_t \psi_t(\vec{x}) = u_t(\psi_t(\vec{x}))$ for all $\vec{x} \in \R^d$.
Hence, if we have the vector field $u$, we can determine the flow $\psi$ by integrating. We therefore now proceed by looking for such a vector field instead of a flow.

Chen et al. \cite{Chen2018NeuralEquations} suggest approximating such a vector field by training a neural network $u^\theta$. Typically, however, we will not have access to a vector field $u$ inducing the desired flow during training: we only have samples from the distributions $\dist{X}_0$ and $\dist{X}_1$. Consequently, the naive flow matching loss
\begin{equation}\label{eq:euclidean_fm_loss}
\mathcal{L}_\mathrm{FM}(\theta) \coloneqq \expectation\left[\norm{u_\rv{T}^\theta (\vecrv{X}_\rv{T}) - u_\rv{T} (\vecrv{X}_\rv{T})}^2\right],
\end{equation}
with $\rv{T} \sim \Uniform[0, 1]$, $\vecrv{X}_0 \sim \dist{X}_0$, and $\vecrv{X}_t \coloneqq \psi_t(\vecrv{X}_0)$, cannot be computed.
Instead, they define a loss on the flow $\psi_1$, which requires simulating the flow during training, making optimisation more complicated and expensive.

\subsubsection{Euclidean Flow Matching.}
To solve this problem, Lipman et al. \cite{Lipman2023FlowModeling,Lipman2024FlowCode} developed Flow Matching (FM). They proposed to condition the vector field on the end point, simply choosing this conditional vector field to be of the form
\begin{equation}\label{eq:euclidean_flow_field}
u_t (\vec{x} \mid \vec{x}_1) \coloneqq \frac{\vec{x}_1 - \vec{x}}{1 - t};
\end{equation}
integrating this vector field will indeed map any starting point $\vec{x}_0$ to the end point $\vec{x}_1$ along the line segment $\vec{x}_t = (1 - t) \vec{x}_0 + t \vec{x}_1$. Then, we can define the following loss function:
\begin{equation}\label{eq:euclidean_cfm_loss}
\mathcal{L}_\mathrm{CFM}(\theta) \coloneqq \expectation[\norm{u_\rv{T}^\theta (\vecrv{X}_\rv{T}) - u_\rv{T} (\vecrv{X}_\rv{T} \mid \vecrv{X}_1)}^2],
\end{equation}
with $\rv{T} \sim \Uniform[0, 1]$, $\vecrv{X}_0 \sim \dist{X}_0$, $\vecrv{X}_1 \sim \dist{X}_1$, and $\vecrv{X}_t \coloneqq (1 - t) \vecrv{X}_0 + t \vecrv{X}_1$ for $t \in [0, 1]$. Note that we \emph{can} compute \eqref{eq:euclidean_cfm_loss}, since we can sample from $\Uniform[0, 1]$, $\dist{X}_0$, and $\dist{X}_1$. It turns out that the gradients (with respect to network parameters $\theta$) of $\mathcal{L}_\mathrm{FM}(\theta)$ and $\mathcal{L}_\mathrm{CFM}(\theta)$ coincide \cite[Thm.~4]{Lipman2024FlowCode}. We can therefore train our network using (stochastic estimates of) the gradient of $\mathcal{L}_\mathrm{CFM}(\theta)$.

\subsubsection{Riemannian Flow Matching.}
However, straight line segments are not well-defined on general Riemannian manifolds.
Consequently, Chen et al. \cite{Chen2024FlowGeometries} generalised this method to FM on Riemannian manifolds. The core principles remain the same, but we now need another way of defining a conditional vector field. The authors found that this can be done by differentiating a premetric. One could use the geodesic distance, yielding geodesics as the integral curves of the conditional vector field. Geodesics are, however, only easy to compute on simple manifolds such as spheres. On other manifolds one must therefore design a tractable premetric, e.g. using spectral distances, and the conditional vector field typically still must be simulated.

\subsubsection{Our Contribution.}
We take an alternative approach: we generalise FM to Lie groups with surjective exponential maps (Thm.~\ref{thm:optimise_on_conditional_loss}), using a conditional flow field whose integral curves are exponential curves (Prop.~\ref{prop:lie_group_flow_field}). This leads to a simple and simulation-free implementation for many Lie groups. On matrix Lie groups the implementation can be particularly straightforward, since the required operations (products, inverses, exponentials, and logarithms) are given by the corresponding matrix operations. Additionally, our method is intrinsic, so all intermediate distributions live on the group by construction.
We show this generalises Euclidean FM \cite{Lipman2023FlowModeling}, by recasting it as FM on the translation group.
As a proof of concept, we performed FM on three Lie groups (Sec.~\ref{sec:experiments}):
\begin{enumerate}
    \item $\SE(2)$: simple group with efficient hand crafted implementation (Fig.~\ref{fig:interpolation_se2}).
    \item $\SO(3)$: matrix group with simple implementation (Fig.~\ref{fig:interpolation_so3}).
    \item $\SE(2) \times \Rtwo$: product group, interesting for generative modelling (Fig.~\ref{fig:interpolation_se2_by_r2}).
\end{enumerate}
FM on Lie groups could be used for more typical generative modelling tasks, e.g. generating images.
Current image generation techniques commonly use variational autoencoders to reduce the ; FM on Lie groups could instead use the more geometrically interpretable latent space afforded by Equivariant Neural Fields \cite{Wessels2025GroundingFields}, consisting of sets of features (in $\R^n$) and poses (in some Lie group).
\section{Lie Group Flow Matching}\label{sec:lie_group_flow_matching}
We first introduce the basic notation for the group operations we use.
\begin{definition}[Lie Group Operations]\label{def:lie_group_operations}
Let $G$ be a Lie group with Lie algebra $\mathfrak{g}$. We denote multiplication of $g_0, g_1 \in G$ by $g_0 g_1$, and the inverse of $g \in G$ by $g^{-1}$. We define the \emph{left action} for any $g \in G$ by $L_g: G \to G; h \mapsto g h$, with push-forward $(L_g)_*$.
We define the \emph{Lie group exponential} by
\begin{equation}\label{eq:exponential_map}
\exp: \mathfrak{g} \to G; A \mapsto \gamma(1), \textrm{ with $\gamma$ the 1-parameter subgroup with } \dot{\gamma}(0) = A.
\end{equation}
If the exponential is surjective,
we can restrict its domain and invert it to find the \emph{Lie group logarithm}:
\begin{equation}\label{eq:logarithm_map}
\log: G \to \mathcal{D}(\exp) \subset \mathfrak{g}; g \mapsto A \textrm{ such that } \exp(A) = g.
\end{equation}
\end{definition}
\begin{remark}
There does not appear to be a complete classification of Lie groups with surjective exponential map. The (complex) general linear group does have a surjective exponential map, but there are subgroups where it fails to be surjective, e.g. the real special linear group \cite{Hall2015LieRepresentations}. However, many Lie groups do have a surjective exponential map, including the special unitary groups $\SU(d)$, the similarity groups $\SIM(d)$, and the Heisenberg groups, in addition to the groups considered in this work.
\end{remark}
Next, we derive FM on Lie groups. We have distributions $\dist{X}_0$ and $\dist{X}_1$ on a Lie group $G$, and look for a flow $\psi: [0, 1] \to \Diff(G)$ such that $(\psi_1)_\# \dist{X}_0 = \dist{X}_1$. This is induced by a time dependent vector field $u: [0, 1] \to \sections(T G)$, satisfying $\partial_t \psi_t(g) = u_t|_{\psi_t(g)}$ for $g \in G$. 
We want to approximate $u$ with a neural network $u^\theta: [0, 1] \to \sections(T G)$, so we should minimise
\begin{equation}\label{eq:lie_group_fm_loss}
\mathcal{L}_\mathrm{FM}^G(\theta) \coloneqq \expectation\left[\norm{u_\rv{T}^\theta (\grouprv{G}_\rv{T}) - u_\rv{T} (\grouprv{G}_\rv{T})}_{\mathcal{G}}^2\right],
\end{equation}
with $\rv{T} \sim \Uniform[0, 1]$, $\grouprv{G}_0 \sim \dist{X}_0$, $\grouprv{G}_t = \psi_t(\grouprv{G}_0)$, and $\mathcal{G}$ some metric tensor field. It is natural to choose $\mathcal{G}$ left-invariant, since then the push-forward of the left action, which can be used to identify tangent spaces with the Lie algebra, is an isometry. It is again impossible to compute the loss in \eqref{eq:lie_group_fm_loss}. We therefore once more introduce a conditional vector field. If the exponential map is surjective, we can always connect $g_0, g_1 \in G$ with an \emph{exponential curve}:
\begin{equation}\label{eq:exponential_curve}
\gamma: [0, 1] \to G; t \mapsto g_0 \exp(t \log(g_0^{-1} g_1)).
\end{equation}
To perform flow matching, we hence choose the conditional vector field such that its integral curves are the exponential curves, in analogy to \eqref{eq:euclidean_flow_field}: 
\begin{proposition}[Lie Group Flow Field]\label{prop:lie_group_flow_field}
The integral curves of the vector field $u_t(\cdot \mid g_1): [0, 1] \to \sections(TG)$, with $g_1 \in G$, given by
\begin{equation}\label{eq:lie_group_flow_field}
u_t(g \mid g_1) = \frac{(L_g)_* \log(g^{-1} g_1)}{1 - t},
\end{equation}
are the exponential curves ending in $g_1$.
\end{proposition}
\begin{proof}
Let $\gamma$ be the exponential curve \eqref{eq:exponential_curve} connecting $g_0, g_1 \in G$. Then, $\gamma$ solves
\begin{equation*}
\begin{cases}
\dot{\gamma}(t) = (L_{\gamma(t)})_* \log(g_0^{-1} g_1), \\
\gamma(0) = g_0.
\end{cases}
\end{equation*}
Noting that $\log(\gamma(t)^{-1} g_1) = (1 - t) \log(g_0^{-1} g_1)$, we see 
\begin{equation*}
\dot{\gamma}(t) = \frac{(L_{\gamma(t)})_* \log(\gamma(t)^{-1} g_1)}{1 - t} = u_t(\gamma(t) \mid g_1),
\end{equation*}
from which we conclude that $\gamma$ is an integral curve of $u(\cdot \mid g_1)$. 
Since $g_0 \in G$ was arbitrary, we have found all integral curves of $u(\cdot \mid g_1)$. \qed
\end{proof}
Then we generalise the loss function \eqref{eq:lie_group_cfm_loss}:
\begin{equation}\label{eq:lie_group_cfm_loss}
\mathcal{L}_\mathrm{CFM}^G(\theta) \coloneqq \expectation[\norm{u_\rv{T}^\theta (\grouprv{G}_\rv{T}) - u_\rv{T} (\grouprv{G}_\rv{T} \mid \grouprv{G}_1)}_{\mathcal{G}}^2],
\end{equation}
with $\rv{T} \sim \Uniform[0, 1]$, $\grouprv{G}_0 \sim \dist{X}_0$, $\grouprv{G}_1 \sim \dist{X}_1$, and $\grouprv{G}_t \coloneqq \grouprv{G}_0 \exp(t \log(\grouprv{G}_0^{-1} \grouprv{G}_1))$.
\begin{remark}
Since we need a Riemannian metric for the loss \eqref{eq:lie_group_cfm_loss}, one might expect that our FM on Lie groups is a specific instance of Riemannian Flow Matching \cite{Chen2024FlowGeometries}. However, we were not able to find a premetric inducing conditional vector field \eqref{eq:lie_group_flow_field}. The logarithmic distance, defined as the length of the exponential curve connecting two points, is the most obvious choice of premetric, but it only gives rise to \eqref{eq:lie_group_flow_field} in specific cases, e.g. when $\mathcal{G}$ is bi-invariant, i.e. invariant under both left and right actions of the group, so that geodesics and exponential curves coincide \cite{Alexandrino2015LieActions}.
\end{remark}
\begin{theorem}[Optimise on Conditional Loss]\label{thm:optimise_on_conditional_loss}
The gradients w.r.t. network parameters $\theta$ of $\mathcal{L}_\mathrm{FM}^G$ \eqref{eq:lie_group_fm_loss} and $\mathcal{L}_\mathrm{CFM}^G$ \eqref{eq:lie_group_cfm_loss} coincide.
\end{theorem}
\begin{proof}
This is a specific case of a general result by Lipman et al. \cite[Prop.~1]{Lipman2024FlowCode}, using that the squared norm $\norm{\cdot}_\mathcal{G}^2$ at a given point $g \in G$ is a Bregman divergence.
\end{proof}

\subsubsection{Reconsidering Euclidean Flow Matching.}
We can now recast Euclidean FM in the Lie group FM framework. On $\R^d$, we have group product $\vec{x} \vec{y} \coloneqq \vec{x} + \vec{y}$, with inverse $\vec{x}^{-1} \coloneqq -\vec{x}$ and identity $e \coloneqq \vec{0}$. It is not hard to see that the Lie group exponential and logarithms are given by $\exp(\vec{x}) = \vec{x}$ and $\log(\vec{x}) = \vec{x}$, respectively. 
Finally, the push-forward of left multiplication is given by $(L_{\vec{x}})_* = \id$. Hence, we can fill in \eqref{eq:lie_group_flow_field} to find \eqref{eq:euclidean_flow_field}:
\begin{equation*}
\frac{(L_{\vec{x}})_* \log(\vec{x}^{-1} \vec{x}_1)}{1 - t} = \frac{\log(\vec{x}_1 - \vec{x})}{1 - t} = \frac{\vec{x}_1 - \vec{x}}{1 - t}.
\end{equation*}
Likewise, the exponential curve \eqref{eq:exponential_curve} reduces to a line segment:
\begin{equation*}
\vec{x}_0 \exp(t \log(\vec{x}_0^{-1} \vec{x}_1)) = \vec{x}_0 + t (\vec{x}_1 - \vec{x}_0) = (1 - t) \vec{x}_0 + t \vec{x}_1,
\end{equation*}
so that we see that the loss function \eqref{eq:lie_group_cfm_loss} reduces to \eqref{eq:euclidean_cfm_loss}.

\subsubsection{Flow Matching on \texorpdfstring{$\SE(2)$}{SE(2)}.}
As an example of a non-Euclidean Lie group, we consider $\SE(2)$:
\begin{definition}[Special Euclidean Group]\label{def:SE2}
We define the 2D \emph{special Euclidean group} as the Lie group $\SE(2) \coloneqq \Rtwo \rtimes \SO(2)$ of roto-translations on two dimensional Euclidean space. Since $\SO(2) \cong S^1$, we can uniquely identify any rotation $R \in \SO(2)$ with an angle $\theta \in \R / 2\pi \Z$. We denote the counter-clockwise rotation with angle $\theta$ by $R_\theta$. The group product is then given by
\begin{equation}\label{eq:SE2_group_product}
(\vec{x}, \theta) (\vec{y}, \phi) = (\vec{x} + R_\theta \vec{y}, \theta + \phi).
\end{equation}
\end{definition}
Note then that we have inverse $(\vec{x}, \theta)^{-1} \coloneqq (-R_\theta^{-1} \vec{x}, -\theta)$ and identity $e \coloneqq (\vec{0}, 0)$. The exponential, with basis $A_1 \coloneqq \partial_x|_e$, $A_2 \coloneqq \partial_y|_e$, $A_3 \coloneqq \partial_\theta|_e$, is given by\footnote{Recall that $\sinc(x) \coloneqq \sin(x) / x$.}
\begin{equation}\label{eq:se2_exponential_map}
\exp(c^i A_i) = \begin{pmatrix}
\sinc(c^3/2) (c^1 \cos(c^3/2) - c^2 \sin(c^3/2)) \\
\sinc(c^3/2) (c^1 \sin(c^3/2) + c^2 \cos(c^3/2)) \\
c^3
\end{pmatrix},
\end{equation}
while the logarithm, with $\mathcal{R}(\log) \coloneqq \mathcal{D}(\exp) \coloneqq \Rtwo \times [-\pi, \pi)$, is given by
\begin{equation}\label{eq:se2_logarithm_map}
\log(\vec{x}, \theta) = \begin{pmatrix}
(x \cos(\theta/2) + y \sin(\theta/2)) / \sinc(\theta/2) \\
(-x \sin(\theta/2) + y \cos(\theta/2)) / \sinc(\theta/2) \\
\theta
\end{pmatrix}.
\end{equation}
We now have all the tools to compute the conditional vector field \eqref{eq:lie_group_flow_field} for $\SE(2)$.

\subsubsection{Flow Matching on Matrix Groups.}
On groups with a matrix representation, we can compute products, inverses, exponentials, and logarithms with the corresponding matrix operations, allowing us to piggy-back on existing implementations, at the cost of requiring more memory compared to a hand-crafted implementation working directly with group elements, as suggested for $\SE(2)$.
In Sec.~\ref{sec:experiments}, we perform experiments on the matrix group $\SO(3)$ as an example.
\begin{definition}[Special Orthogonal Group]\label{def:SO3}
We define the 3D \emph{special orthogonal group} as the Lie group $\SO(3)$ of origin-preserving rotations on three dimensional Euclidean space. We can represent $\SO(3)$ with $3 \times 3$ orthogonal matrices with determinant $1$.
\end{definition}
We have implemented flow matching using \texttt{PyTorch} \cite{Ansel2024PyTorch}, which contains methods for matrix multiplication, inverses, and exponentials. It does not contain a matrix logarithm, however; for $\SO(3)$ we can use Rodrigues' formula:
\begin{equation}\label{eq:so3_logarithm_map}
\log(R) \coloneqq \sinc(q) \frac{R - R^T}{2}, \textrm{ with } q \coloneqq \arccos\left(\frac{\tr(R) - 1}{2}\right).
\end{equation}
Since $\SO(3)$ is compact, it can be equipped with a bi-invariant metric. We hence recover Riemannian FM \cite{Chen2024FlowGeometries}, as the geodesics and exponential curves coincide \cite{Hall2015LieRepresentations}.

\subsubsection{Flow Matching on Product Groups.}
If we can perform flow matching on $G$ and $H$, then we can also do so on $(G \times H)^m$ for $m \in \N$: all the relevant operations are inherited from $G$ and $H$. In particular, this means we can perform flow matching on $(\SE(2) \times \R^d)^m$, the space in which e.g. latent codes of Equivariant Neural Fields live \cite{Wessels2025GroundingFields}. We have performed experiments on $\SE(2) \times \Rtwo$, see Sec.~\ref{sec:experiments}.

\subsubsection{Flow Matching on Homogeneous Spaces.}
We call a manifold $\mathcal{M}$ a homogeneous space of Lie group $G$ if $G$ acts transitively on $\mathcal{M}$, which is to say that for any pair of points $p_0, p_1 \in \mathcal{M}$ we can find $g \in G$ such that $g p_0 = p_1$. This allows us to connect $p_0$ and $p_1$ with the curve $\gamma(t) = \exp(t \log(g)) p_0$, which is a projection of an exponential curve in $G$ onto $\mathcal{M}$. Hence, our framework can be generalised to work with homogeneous spaces too.

One difficulty is that there are typically infinitely many $g$ such that $g p_0 = p_1$, and so infinitely many exponential curves. Hence, one must find a way of selecting a single curve. This has been investigated e.g. for the $\SE(3)$ homogeneous space of three dimensional positions and orientations $\R^3 \times S^2$: there is a computationally convenient choice with links to left-invariant distance approximations \cite{Portegies2015NewNeuroimaging}.
\section{Experiments}\label{sec:experiments}
Here we show experiments performed with three groups: $\SE(2)$, $\SO(3)$, and $\SE(2) \times \Rtwo$. The implementations and animations of the flows are available at \url{https://github.com/finnsherry/FlowMatching}.
Recall that we must learn a time dependent vector field $u^\theta: [0, 1] \to \sections(T G)$. In practice, we train a multilayer perceptrons with four hidden layers and a width of 64 neurons to map a group element $g \in G$ and time $t \in [0, 1]$ to the components (in $\R^{\dim G}$) of a vector in the tangent space at $T_g G$ with respect to a fixed left-invariant frame. We can then integrate the learned flow using the Lie group exponential. As a consequence, the flow remains on the group, without needing to impose constraints on the network or project back onto the manifold as in \cite{Chen2024FlowGeometries}.

We can identify $\SE(2) \cong \Rtwo \times S^1$, the space of planar positions and orientations. Similarly, we can identify $\SO(3)$ with the space of spherical positions and orientations, which is a non-trivial fibre bundle over $S^2$ with typical fibre $S^1$ 
(for details on these spaces of positions and orientations, see \cite{Berg2025CrossingImages}).
This means that we can visualise points in $\SE(2)$ and $\SO(3)$ as arrows on the plane and sphere, respectively. For points in $\SE(2) \times \Rtwo$, we can simply separately plot the $\SE(2)$ and $\Rtwo$ components. 

\begin{figure}
\centering
\includegraphics[width=\linewidth, trim={0 0.25cm 0 0.25cm}, clip]{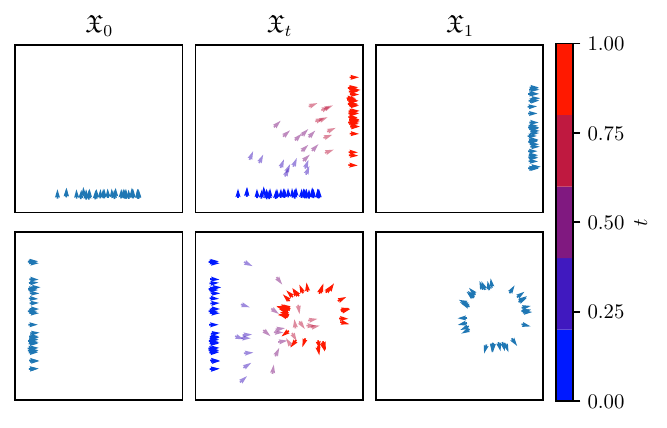}
\caption{FM on $\SE(2)$, interpreted as the space of planar positions and orientations. Top: flowing from horizontal line to vertical line. Bottom: flowing from vertical line to circle.}\label{fig:interpolation_se2}
\end{figure}

\begin{figure}
\centering
\includegraphics[width=\linewidth, trim={0 0.2cm 0 0.2cm}, clip]{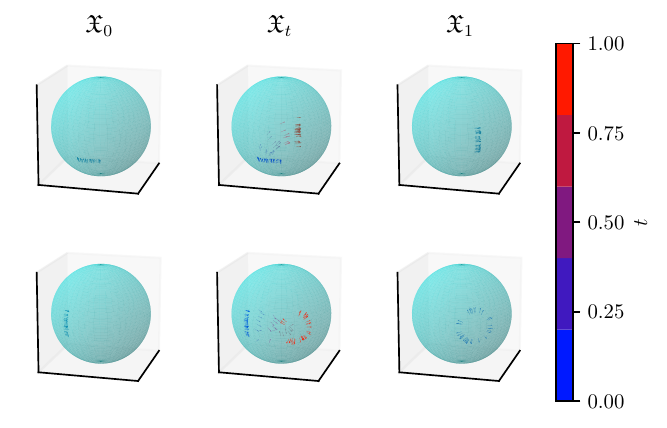}
\caption{FM on $\SO(3)$, interpreted as the space of spherical positions and orientations. Top: flowing from horizontal line to vertical line. Bottom: flowing from vertical line to circle.}\label{fig:interpolation_so3}
\end{figure}

\begin{figure}
\centering
\includegraphics[width=\linewidth, trim={0 0.25cm 0 0.25cm}, clip]{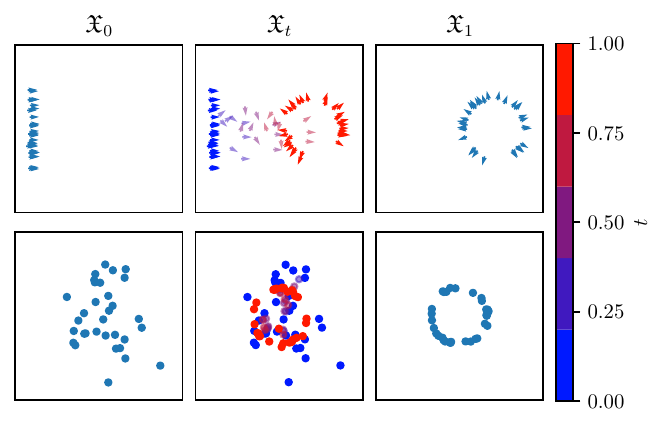}
\caption{FM on $\SE(2) \times \Rtwo$, with $\SE(2)$ interpreted as the space of planar positions and orientations. Note that this shows the flow of a single pair of distributions $\dist{X}_0$ and $\dist{X}_1$: the top row shows the $\SE(2)$ component and bottom row shows the $\Rtwo$ component.}\label{fig:interpolation_se2_by_r2}
\end{figure}

Figs.~\ref{fig:interpolation_se2}, \ref{fig:interpolation_so3}, \ref{fig:interpolation_se2_by_r2} show FM on $\SE(2)$, $\SO(3)$, and $\SE(2) \times \Rtwo$, respectively. In each case, the left column shows samples from the initial distribution $\dist{X}_0$ and the right column shows samples from the target distribution $\dist{X}_1$. In the centre column, we take samples from $\dist{X}_0$ (blue) and flow them forward (transparent); if the network has been trained successfully, the samples at $t = 1$ (red) should appear to be sampled from $\dist{X}_1$. For $\SE(2)$ and $\SO(3)$, the rows show different pairs of distributions $\dist{X}_0$ and $\dist{X}_1$; for $\SE(2) \times \Rtwo$, we have a single pair of distributions $\dist{X}_0$ and $\dist{X}_1$, and the rows show the $\SE(2)$ and $\Rtwo$ components.

In all cases, the samples at $t = 1$ indeed reasonably match the target distribution. In simple cases, where we flow from lines to lines, the interpolants $\dist{X}_t$ also behave nicely. For more complicated cases, where we flow from a line to a circle, the interpolants look messier, which is unsurprising, as the exponential curves connecting samples in $\dist{X}_0$ and $\dist{X}_1$ can be quite intricate.

\subsubsection{Conclusion \& Future Work.}
We generalised FM to Lie groups with surjective exponential maps, using a conditional flow field whose integral curves are exponential curves (Prop.~\ref{prop:lie_group_flow_field}). 
This has an intrinsic, simple, and simulation-free implementation for many Lie groups.  
As a proof of concept, we performed FM on three Lie groups (Sec.~\ref{sec:experiments}).
FM on Lie groups could be used for generative modelling with data consisting of sets of features (in $\R^n$) and poses (in some Lie group), e.g. the latent codes of Equivariant Neural Fields \cite{Wessels2025GroundingFields}.

\subsubsection{Acknowledgements.}
EAISI is gratefully acknowledged for financial support through the EIDMAR programme.
The European Commission is gratefully acknowledged for financial support through HORIZON-MSCA-2020-SE project REMODEL.

\bibliographystyle{splncs04}
\bibliography{references}
\end{document}